\documentclass[12pt]{amsart}
\usepackage{a4}
\usepackage{amssymb}

\makeatletter
\@addtoreset{equation}{section}
\makeatother

\usepackage{color}
\usepackage[mathscr]{eucal}
\usepackage{amsmath,amsthm,amssymb}
\usepackage{mathrsfs}
\usepackage{enumerate}
\usepackage{bm}
\usepackage{graphicx}
\usepackage{verbatim}
\usepackage{wrapfig}
\usepackage{ascmac}
\usepackage{multicol}

\usepackage{hyperref}
\newtheorem{Prop}{Proposition}[section]
\newtheorem{Thm}[Prop]{Theorem}

\newtheorem{conj}[Prop]{Conjecture}

\theoremstyle{definition}
\newtheorem{Def}[Prop]{Definition}
\newtheorem{Ex}[Prop]{Example}
\newtheorem{Rem}[Prop]{Remark}
\newtheorem{Problem}[Prop]{Problem}

\newcommand{\R}{{\mathbb R}}

\newcommand{\Z}{{\mathbb Z}}

\newcommand{\RP}{\R \mathrm{P}}

\newcommand{\Map}{\mathrm{Map}}
\newcommand{\id}{\mathrm{id}}

\newcommand{\Fix}{\mathrm{Fix}}

\newcommand{\Dis}{\mathrm{Dis}}

\newcommand{\rnum}[1]{\expandafter{\romannumeral #1}}

\title{Two-numbers and Euler characteristics for quandles} 

\author{Ryoya Kai}
\address[R.~Kai]{Nara University of Education, Nara City, Japan 630-8528} 
\email{kai.ryoya.d8@cc.nara-edu.ac.jp}

\author{Akira Kubo}
\address[A.~Kubo]{Faculty of Environmental Studies, Hiroshima Institute of Technology, Hiroshima City, Japan 731-5143}
\email{a.kubo.3r@cc.it-hiroshima.ac.jp}

\author{Hiroshi Tamaru}
\address[H.~Tamaru]{Department of Mathematics, Osaka Metropolitan University, Osaka City, Japan 558-8585} 
\email{tamaru@omu.ac.jp}

\thanks{
} 

\date{}

\makeatletter
\@namedef{subjclassname@2020}{\textup{2020} Mathematics Subject Classification}
\makeatother

\begin{document} 

\maketitle 

\begin{abstract} 
Quandles can be regarded as generalizations of symmetric spaces. 
In the theory of symmetric spaces developed by Chen and Nagano, 
there is an interesting relationship between the two-number and the Euler characteristic. 
The two-number is a Riemannian geometric invariant that can also be characterized in terms of point symmetries, 
whereas the Euler characteristic is a topological invariant. 
The aim of this paper is to initiate the study of quandle analogues of Chen--Nagano theory. 
In particular, 
we investigate relationships between the two-numbers and the Euler characteristics of quandles, 
and provide examples of finite quandles that either satisfy or fail to satisfy properties analogous 
to those of symmetric spaces. 
These examples are constructed from directed simple graphs labeled by abelian groups. 
\end{abstract}

\section{Introduction}
\label{sec:1} 

The notion of a quandle originated in knot theory \cite{Joyce, Matveev}, 
and now plays important roles in many branches of mathematics. 
Research on quandles can be broadly divided into two directions: 
one is to investigate the structures of quandles from various viewpoints, 
and the other is to apply results from quandle theory to other fields. 
A typical example is the interaction between quandle theory and knot theory, 
although such interactions are no longer limited to knot theory. 
The aim of our study is to develop a similar interaction between quandle theory and the theory of symmetric spaces 
(\cite{Tamaru, IT}). 
We therefore adopt notation analogous to that used for symmetric spaces, as follows. 

\begin{Def}
Let $X$ be a nonempty set, and denote by $\Map(X,X)$ the set of all maps from $X$ to itself. 
A pair $(X,s)$, where
\[
s \colon X \rightarrow \Map(X,X), \quad x \mapsto s_x , 
\]
is called a \textit{quandle} if the following conditions hold:
\begin{itemize}
\item[(Q1)]
$s_x(x)=x$ for every $x\in X$;
\item[(Q2)]
$s_x$ is bijective for every $x\in X$;
\item[(Q3)]
$s_x\circ s_y=s_{s_x(y)}\circ s_x$ for every $x,y\in X$.
\end{itemize}
\end{Def}

Note that a symmetric space carries a quandle structure induced by its point symmetries 
(cf.~\cite{Helgason, Loos}). 
For a quandle $(X,s)$, the map 
$s_x \colon X \to X$ 
is also called the \textit{point symmetry} at $x \in X$, although it need not be an involution.

%%%

In the theory of symmetric spaces, 
Chen and Nagano developed a remarkable theory. 
For a symmetric space $(X,s)$, 
an invariant called the \textit{two-number}, denoted by $\#_2(X,s)$,
was defined in terms of Riemannian geometric data, namely, closed geodesics. 
Its definition can also be reformulated solely in terms of point symmetries. 
Remarkably, this invariant is closely related to the topology of the symmetric space.
In particular, for every compact connected Riemannian symmetric space $(X,s)$,
the following inequality involving the topological Euler characteristic holds:
\[
\#_2(X,s) \geq \chi^{\mathrm{top}}(X),
\]
which we refer to as the \textit{Chen--Nagano inequality}.
It is also conjectured that, for every compact connected Riemannian symmetric space,
the two-number and the Euler characteristic have the same parity.
Further details of the Chen--Nagano theory will be presented in Section~\ref{sec2}.

%%%

Our aim is to establish a quandle analogue of the Chen--Nagano theory.
As mentioned above,
the two-number of a symmetric space can be characterized in terms of its point symmetries.
This characterization allows us to define the corresponding notion for quandles \cite{KNOT}.
Moreover, the Euler characteristic of a compact connected Riemannian symmetric space
can also be formulated in terms of its point symmetries,
and this formulation leads to an analogous invariant for quandles \cite{KT}.
It is therefore natural to consider the following properties.

\begin{Def}
\label{def:0822}
Let $(X,s)$ be a quandle,
and denote its two-number by $\#_2(X,s)$
and its quandle Euler characteristic by $\chi^{\mathrm{qdl}}(X,s)$.
We say that $(X,s)$ satisfies
\begin{enumerate}
\item
the \textit{Chen--Nagano-type inequality} if
\[
\#_2(X,s) \geq \chi^{\mathrm{qdl}}(X,s);
\]
\item
the \textit{parity agreement} if
\[
\#_2(X,s) \equiv \chi^{\mathrm{qdl}}(X,s) \pmod{2}.
\]
\end{enumerate}
\end{Def}

In this paper,
we construct examples of finite quandles from directed simple graphs labeled by abelian groups. 
We compute the two-numbers and the quandle Euler characteristics of some of these quandles,
including those for which the underlying graph is the Johnson graph $J(5,2)$.
This particular case may also be of independent interest.
Our examples demonstrate that neither the Chen--Nagano-type inequality
nor the parity agreement holds for quandles in general.
It would therefore be interesting to identify conditions under which a quandle satisfies these properties.
This problem may be viewed as part of an effort to identify a natural subclass of quandles
that shares certain properties with symmetric spaces.

\section{Two-numbers and the Chen--Nagano theory}
\label{sec2}

In this section,
we briefly review the Chen--Nagano theory of compact symmetric spaces,
with particular emphasis on two-numbers.
The notion of the two-number was introduced by Chen and Nagano
for symmetric spaces \cite{CN}.
Since its characterization in terms of point symmetries extends naturally to quandles,
we begin by defining the two-number of a quandle.

\begin{Def}
Let $(X,s)$ be a quandle.
A subset $A\subset X$ is said to be \textit{antipodal} if $s_a(b)=b$ holds for every $a,b\in A$. 
The \textit{two-number} of $(X,s)$ is defined by 
\[
\#_2(X,s)
:=
\sup\bigl\{
\#A
\bigm|
A\subset X \text{ is antipodal}
\bigr\}.
\]
\end{Def}

From the viewpoint of quandle theory,
an antipodal subset is precisely a subquandle whose induced quandle structure is trivial.
Here, a quandle is said to be \textit{trivial}
if all of its point symmetries are the identity map.
The term ``antipodal subset'' originates from the usual notion of antipodal points on a sphere,
as illustrated by the following examples.

\begin{Ex}
\label{ex:0822}
We have the following: 
\begin{enumerate}
\item
For the standard sphere $S^n$, 
every maximal antipodal subset is of the form $\{ p , -p \}$ for some $p\in S^n$. 
Consequently, 
\[
\#_2(S^n)=2.
\]
\item
For the real projective space $\RP^n$, every maximal antipodal subset is congruent to 
$\{ [e_1], \ldots , [e_{n+1}] \}$, 
where $\{e_1,\ldots,e_{n+1}\}$ is an orthonormal basis of $\R^{n+1}$. 
Consequently,
\[
\#_2(\RP^n)=n+1.
\]
\end{enumerate}
\end{Ex}

\begin{proof}
These assertions are well known, 
but we include a brief proof for the convenience of the reader. 
Recall that the point symmetry of the standard sphere $S^n$ at $x\in S^n$ 
is the restriction to $S^n$ of the orthogonal linear transformation 
\[
s_x^{S^n}(y)
=
2\langle x,y\rangle x-y.
\]
This transformation acts as the identity on the line $\R x$ 
and as $-\id$ on its orthogonal complement $(\R x)^\perp$. 
Its fixed-point set in $S^n$ is therefore $\{x,-x\}$. 
It follows that every maximal antipodal subset of $S^n$ is of the form $\{x,-x\}$, 
which proves the first assertion. 

For the real projective space $\RP^n$,
the point symmetry is induced by the corresponding point symmetry of $S^n$:
\[
s_{[x]}^{\RP^n}([y])
=
\bigl[s_x^{S^n}(y)\bigr].
\]
Hence $s_{[x]}^{\RP^n}$ fixes $[y]$ if and only if
$y$ is proportional to $x$ or orthogonal to $x$.
Thus, an antipodal subset of $\RP^n$ is represented by a set of pairwise orthogonal lines in $\R^{n+1}$.
Such a set contains at most $n+1$ elements,
and equality is attained by the lines spanned by the vectors of an orthonormal basis.
This proves the second assertion.
\end{proof}

\begin{Rem}
We summarize several basic properties of the two-number;
see \cite{CN} for the corresponding results on symmetric spaces.
\begin{itemize} 
\item 
The two-number of every compact connected Riemannian symmetric space is finite. 
\item 
For quandles, 
the two-number $\#_2$ is an isomorphism invariant. 
Moreover,
it provides an obstruction to the existence of quandle embeddings. 
Indeed, 
if there exists a quandle embedding, that is, an injective quandle homomorphism 
\[
f \colon (X,s^X)\longrightarrow (Y,s^Y),
\]
then it satisfies 
\[
\#_2(X,s^X)\leq \#_2(Y,s^Y).
\]
To see this,
observe that the image under $f$ of any antipodal subset of $X$ is an antipodal subset of $Y$. 
\item 
As an application,
for $n\geq 2$,
the symmetric-space quandle $\RP^n$ cannot be embedded into
the symmetric-space quandle $S^m$ for any $m$,
because
\[
\#_2(\RP^n)=n+1>2=\#_2(S^m).
\]
\end{itemize} 
\end{Rem}

One of the main achievements of the Chen--Nagano theory
is the establishment of relationships between the two-number $\#_2(M)$
and the topology of a compact connected Riemannian symmetric space $M$.
We recall a statement concerning the topological Euler characteristic $\chi^{\mathrm{top}}(M)$. 

\begin{Thm}[Chen--Nagano \cite{CN}]
Let $M$ be a compact connected Riemannian symmetric space. 
Then it satisfies 
\[
\#_2(M) \geq \chi^{\mathrm{top}}(M).
\]
\end{Thm}

We refer to this inequality as the \textit{Chen--Nagano inequality}.
Further results concerning this inequality can be found in \cite{CN}.
For example,
equality holds when $M$ is a compact Hermitian symmetric space of semisimple type.
The relationship between the two-number and the Euler characteristic
also motivates the following conjecture.

\begin{conj}[cf.\ \cite{Chen}]
Let $M$ be a compact connected Riemannian symmetric space.
Then it satisfies 
\[
\#_2(M) \equiv \chi^{\mathrm{top}}(M) \pmod{2}.
\]
\end{conj}

We refer to this conjecture as the \textit{Chen--Nagano conjecture}.
Compact connected irreducible Riemannian symmetric spaces have been classified,
and computations based on this classification suggest that the conjecture is valid.
Nevertheless,
a conceptual proof is desirable,
as it may clarify the underlying mechanism responsible for the parity agreement.

We conclude this section by verifying the Chen--Nagano conjecture
for the standard sphere $S^n$ and the real projective space $\RP^n$.
Their Euler characteristics are well known,
and their two-numbers were computed in Example~\ref{ex:0822}.
Combining these facts yields the following examples.

\begin{Ex}
For the standard sphere $S^n$, we have
\[
\#_2(S^n)=2,
\qquad
\chi^{\mathrm{top}}(S^n)
=
\begin{cases}
0 & \text{(if $n$ is odd)},\\
2 & \text{(if $n$ is even)}.
\end{cases}
\]
Thus, both $\#_2(S^n)$ and $\chi^{\mathrm{top}}(S^n)$ are even. 
Consequently, $S^n$ satisfies the Chen--Nagano conjecture. 
\end{Ex}

\begin{Ex}
For the real projective space $\RP^n$, we have
\[
\#_2(\RP^n)=n+1,
\qquad
\chi^{\mathrm{top}}(\RP^n)
=
\begin{cases}
0 & \text{(if $n$ is odd)},\\
1 & \text{(if $n$ is even)}.
\end{cases}
\] 
If $n$ is odd, then both $\#_2(\RP^n)=n+1$ and $\chi^{\mathrm{top}}(\RP^n)=0$ are even. 
If $n$ is even, then both $\#_2(\RP^n)=n+1$ and $\chi^{\mathrm{top}}(\RP^n)=1$ are odd. 
Consequently, $\RP^n$ satisfies the Chen--Nagano conjecture. 
\end{Ex}

%%%

\section{Euler characteristics for quandles} 

Our aim is to establish a quandle analogue of the Chen--Nagano theory.
Recall that the two-number has already been defined for quandles.
To formulate an analogue of the Chen--Nagano theory,
we also need an appropriate notion of the Euler characteristic for quandles.
Such a notion was introduced by the first and third authors in \cite{KT}.
In this section,
we recall its definition, basic properties, and several examples.
We first introduce the displacement group of a quandle.

\begin{Def}
The \textit{displacement group} of a quandle $(X,s)$ is defined by
\[
\Dis(X,s)
:=
\left\langle
s_x\circ s_y^{-1}
\mathrel{\Big|}
x,y\in X
\right\rangle_{\mathrm{grp}}.
\]
\end{Def}

When a compact connected Riemannian symmetric space $M$ is regarded as a quandle,
its displacement group is a compact connected Lie group
that acts transitively on $M$.
Using the action of the displacement group,
we define the Euler characteristic of a quandle as follows.

\begin{Def}[\cite{KT}]
For $g\in\Dis(X,s)$, let
\[
\Fix(g)
:=
\{x\in X\mid g(x)=x\}.
\]
The \textit{quandle Euler characteristic} of a quandle $(X,s)$ is defined by
\[
\chi^{\mathrm{qdl}}(X,s)
:=
\inf
\left\{
\#\Fix(g)
\mathrel{\Big|}
g\in\Dis(X,s)
\right\}.
\]
\end{Def}

This definition is motivated by a classical description of the Euler characteristic of a compact homogeneous space. 
More precisely,
if $G/K$ is a homogeneous space with $G$ compact and connected,
then its topological Euler characteristic can be described
in terms of the fixed-point set of the action of a maximal torus of $G$.
Using this description,
one obtains the following compatibility result.

\begin{Thm}[\cite{KT}]
Let $M$ be a compact connected Riemannian symmetric space,
regarded as a quandle via its point symmetries.
Then
\[
\chi^{\mathrm{qdl}}(M)
=
\chi^{\mathrm{top}}(M).
\]
\end{Thm}

We now describe two examples,
namely, the standard sphere $S^n$ and the real projective space $\RP^n$.
These examples illustrate that their Euler characteristics
coincide with the cardinalities of the fixed-point sets
of the actions of maximal tori.

\begin{Ex}
The standard sphere $S^n$ can be expressed as 
\[
S^n = \mathsf{SO}(n+1)/\mathsf{SO}(n) . 
\]
Let $T$ be a maximal torus of $\mathsf{SO}(n+1)$. 
Then the following statements hold: 
\begin{enumerate}
\item
If $n$ is odd, then
\[
T\cong \mathsf{SO}(2)^{(n+1)/2},
\]
and the action of $T$ on $S^n$ has no fixed points.

\item
If $n$ is even, then
\[
T\cong \mathsf{SO}(2)^{n/2},
\]
and the action of $T$ on $S^n$ has exactly two fixed points.
\end{enumerate}
In both cases, $\#\Fix(T,S^n)$, 
the cardinality of the fixed point set of the action of $T$, 
coincides with $\chi^{\mathrm{top}}(S^n)$. 
\end{Ex}

Let us describe these maximal tori more explicitly.
For $n=5$ and $n=6$, they can be realized as the following block-diagonal subgroups:
\begin{align*}
\mathsf{SO}(6)
&\supset
\left\{
\begin{pmatrix}
A_1 & & \\
& A_2 & \\
& & A_3
\end{pmatrix}
\mathrel{\Big|}
A_1,A_2,A_3\in\mathsf{SO}(2)
\right\},
\\
\mathsf{SO}(7)
&\supset
\left\{
\begin{pmatrix}
1 & & & \\
& A_1 & & \\
& & A_2 & \\
& & & A_3
\end{pmatrix}
\mathrel{\Big|}
A_1,A_2,A_3\in\mathsf{SO}(2)
\right\}.
\end{align*}
In the first case,
there is no nonzero vector fixed by the entire torus,
and hence its action on $S^5$ has no fixed points.
In the second case,
the one-dimensional subspace corresponding to the first coordinate
is fixed pointwise by the entire torus.
Its intersection with $S^6$ consists of two points,
and hence the action has exactly two fixed points.

\begin{Ex}
The real projective space $\RP^n$ can be expressed as
\[
\RP^n
=
\mathsf{SO}(n+1)/
\mathsf{S}\bigl(\mathsf{O}(1)\times\mathsf{O}(n)\bigr).
\]
A maximal torus $T$ of $\mathsf{SO}(n+1)$
can be chosen as in the preceding example.
If $n$ is odd,
then its action on $\RP^n$ has no fixed points.
If $n$ is even,
then the one-dimensional subspace fixed by $T$
determines exactly one fixed point in $\RP^n$.
Consequently $\#\Fix(T,\RP^n)$, 
the cardinality of the fixed point set of the action of $T$, 
coincides with $\chi^{\mathrm{top}}(\RP^n)$. 
\end{Ex}

We note here that the Euler characteristics of compact connected Riemannian symmetric spaces are completely determined.
We summarize some relevant results below.

\begin{Rem}
Let $M=G/K$ be a compact connected Riemannian symmetric space, 
where $G$ is the identity component of the isometry group of $M$. 
Then the following statements hold. 
\begin{enumerate}
\item
The Euler characteristic $\chi^{\mathrm{top}}(G/K)$ vanishes if and only if
\[
\mathrm{rk}(G)>\mathrm{rk}(K).
\]
Indeed, as mentioned above, 
the Euler characteristic of $G/K$ coincides with the cardinality of the fixed-point set of the action of 
a maximal torus $T$ of $G$. 
The fixed-point set is nonempty if and only if $T$ is conjugate to a subgroup of $K$, which is equivalent to
$\mathrm{rk}(G)=\mathrm{rk}(K)$. 

\item
If $\chi^{\mathrm{top}}(G/K)\neq 0$, then it can be computed in terms of the Weyl groups of $G$ and $K$; see \cite{Wang}.
More precisely,
\[
\chi^{\mathrm{top}}(G/K)
=
\frac{\#W(G)}{\#W(K)}.
\]

\item
If $\chi^{\mathrm{top}}(G/K)\neq 0$, then it can also be computed inductively by using certain totally geodesic submanifolds, called polars; see \cite{Nagano}.
\end{enumerate}
\end{Rem}

We have now introduced the definitions necessary
to formulate a quandle analogue of the Chen--Nagano theory.
In Definition~\ref{def:0822},
we introduced the following two properties of a quandle $(X,s)$:
\begin{itemize}
\item
the \textit{Chen--Nagano-type inequality}:
\[
\#_2(X,s)\geq\chi^{\mathrm{qdl}}(X,s);
\]
\item
the \textit{parity agreement}:
\[
\#_2(X,s)
\equiv
\chi^{\mathrm{qdl}}(X,s)
\pmod{2}.
\]
\end{itemize}
In the next section, 
we construct examples of quandles that satisfy these properties and examples that fail to satisfy them. 

\section{Quandles associated with weighted graphs} 

In this section, we present a method for constructing quandles from directed simple graphs 
whose edges are labeled by elements of abelian groups. 
This construction is a slight generalization of those given in \cite{FT,SS}. 

Recall that a directed simple graph is a directed graph with neither self-loops nor multiple edges.
More precisely, if $V$ is a vertex set, then the edge set $E$ satisfies
\[
E \subset \{ (v,w) \in V \times V \mid v \neq w \}.
\]

\begin{Prop}
Let $(V,E)$ be a finite directed simple graph whose vertex set is given by
\[
V=\{A_1,\ldots,A_n\},
\]
where each $A_i$ is an abelian group written additively.
Assume that each edge $(A_i,A_j)\in E$ is labeled by an element $w(i,j)\in A_j$.
If $(A_i,A_j)\notin E$, then we define $w(i,j):=0$.
Then the disjoint union
\[
X:=A_1 \sqcup \cdots \sqcup A_n
\]
admits a quandle structure defined by
\[
s_{a_i}(a_j)=a_j+w(i,j)
\qquad
(a_i\in A_i,\ a_j\in A_j).
\]
\end{Prop}

\begin{proof}
It suffices to verify the three quandle axioms.
The first axiom (Q1) follows immediately from the fact that the graph $(V,E)$ has no self-loops. 
Hence $(A_i,A_i)\notin E$, and therefore $w(i,i)=0$ for every $i$. 
The second axiom (Q2) is also immediate, since the inverse of $s_{a_i}$ is given by 
\[
(s_{a_i})^{-1}(a_j)=a_j-w(i,j)
\qquad
(a_i\in A_i,\ a_j\in A_j).
\]
To verify the third axiom (Q3), let $a_i\in A_i$, $a_j\in A_j$, and $a_k\in A_k$.
By definition,
\begin{align*}
s_{a_i}\circ s_{a_j}(a_k)
&=
s_{a_i}\bigl(a_k+w(j,k)\bigr) \\
&=
a_k+w(j,k)+w(i,k),
\\
s_{s_{a_i}(a_j)}\circ s_{a_i}(a_k)
&=
s_{a_j+w(i,j)}\bigl(a_k+w(i,k)\bigr) \\
&=
a_k+w(i,k)+w(j,k).
\end{align*}
Since $A_k$ is abelian, the two expressions coincide.
Therefore axiom (Q3) holds, which completes the proof.
\end{proof}

Clearly, the quandles obtained by this construction have a rather special form.
Nevertheless, this method yields many interesting examples of quandles, which may serve as useful test cases and potential sources of counterexamples.
Although the abelian groups $A_i$ may be arbitrary, we shall mainly consider the case in which they are cyclic groups
\[
\Z_n := \Z / n\Z.
\]

\begin{Rem}
We make the following observations concerning the construction above.
\begin{enumerate}
\item
Every quandle obtained by this construction has an abelian inner automorphism group.
Indeed, all point symmetries commute with one another.

\item
Our construction is a slight generalization of that of Saito and Sugawara \cite{SS}.
Their construction corresponds to the special case in which all the abelian groups $A_i$ are mutually isomorphic.
This special case is sufficient for constructing homogeneous quandles with abelian inner automorphism groups.

\item
The construction of Saito and Sugawara \cite{SS} is itself a generalization of that of Furuki and the third author \cite{FT}.
The latter construction corresponds to the special case in which all the groups $A_i$ are isomorphic to $\Z_2$ and the directed graph is bidirected.
\end{enumerate}
\end{Rem}

We now consider quandles constructed from certain specific graphs and determine their two-numbers and quandle Euler characteristics.

\begin{Ex}
Consider the quandle $(X,s)$ obtained from the graph shown in Figure~\ref{fig:graph-two-vertices}. 
\begin{figure}[htbp]
\centering
\includegraphics[width=0.3\textwidth]{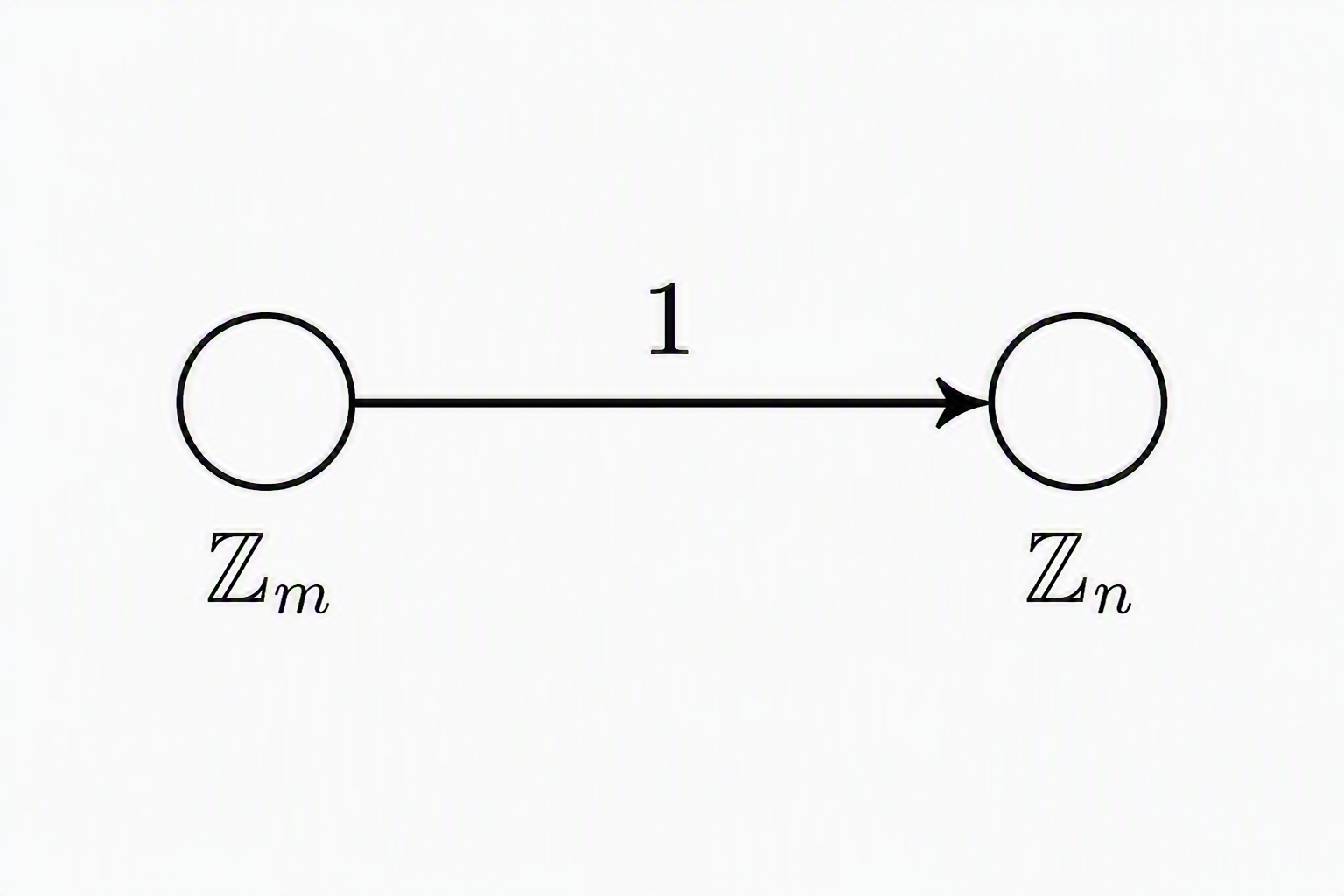}
\caption{A graph with two vertices.}
\label{fig:graph-two-vertices}
\end{figure}
Then the following hold:
\begin{enumerate}
\item
$\#_2(X,s)=\max\{m,n\}$.
\item
$\chi^{\mathrm{qdl}}(X,s)=m$.
\end{enumerate}
Consequently, this quandle satisfies the Chen--Nagano-type inequality, 
but it does not satisfy the parity agreement in general.
\end{Ex}

\begin{proof}
We have
\[
X=\Z_m\sqcup\Z_n.
\]
For every $x\in\Z_m$, the point symmetry $s_x$ acts trivially on $\Z_m$ and as translation by $1$ on $\Z_n$.
On the other hand, for every $y\in\Z_n$, the point symmetry $s_y$ is the identity map on $X$.
Therefore, a maximal antipodal subset is either $\Z_m$ or $\Z_n$, and hence
\[
\#_2(X,s)=\max\{m,n\}.
\]

Moreover, the above description of the point symmetries shows that the
displacement group is generated by the elements
\[
s_x = s_x \circ s_y^{-1} \in \Dis(X,s)
\qquad
(x \in \mathbb{Z}_m,\ y \in \mathbb{Z}_n).
\]
Each of these elements acts nontrivially only on $\mathbb{Z}_n$.
Consequently, the fixed-point set of every nonidentity element of the
displacement group is precisely $\mathbb{Z}_m$. Therefore,
\[
\chi^{\mathrm{qdl}}(X,s)=m,
\]
which proves the second assertion. 
\end{proof}

\begin{Ex}
Consider the quandle $(X,s)$ obtained from the graph shown in Figure~\ref{fig:graph-three-vertices}.
\begin{figure}[htbp]
\centering
\includegraphics[width=0.3\textwidth]{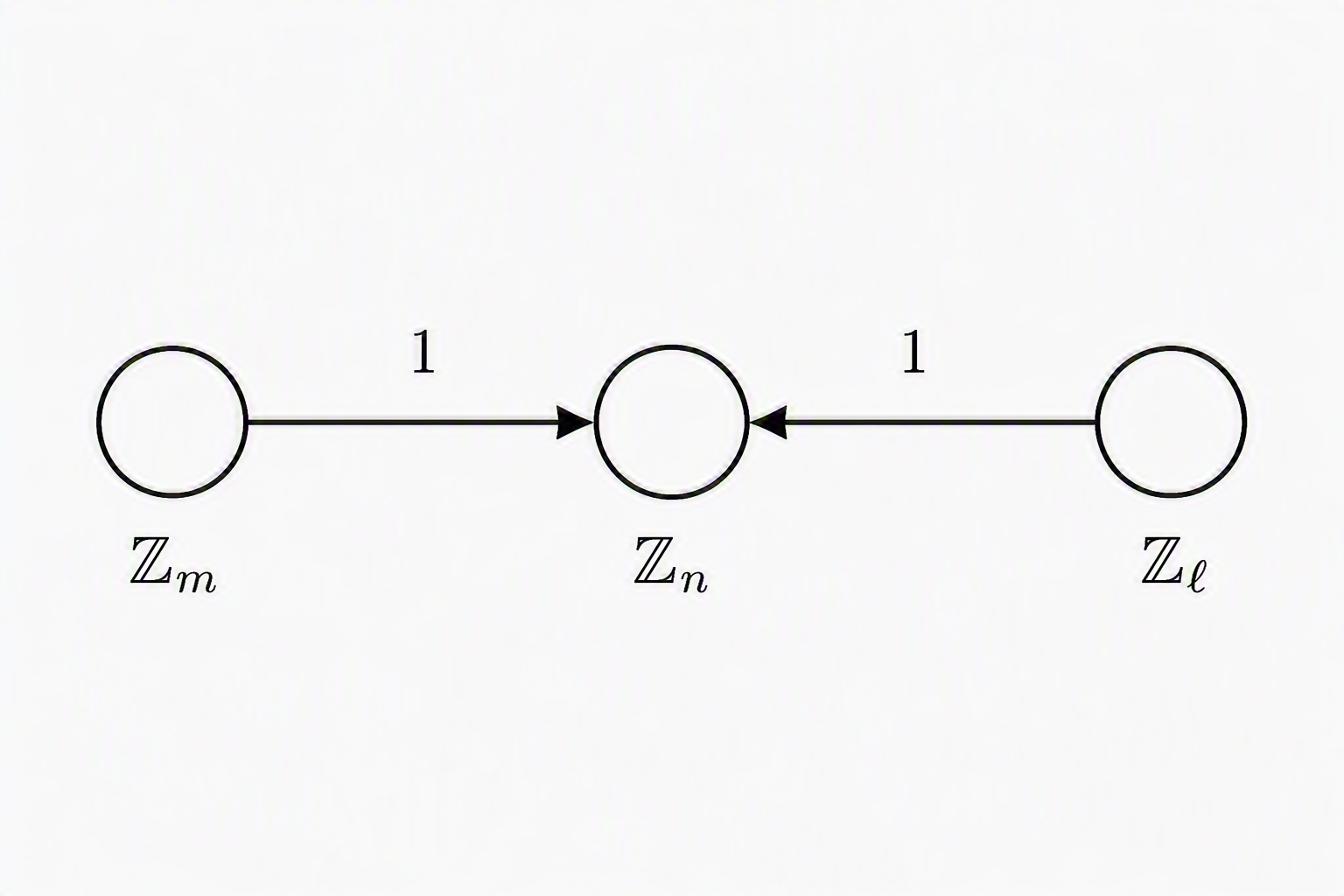}
\caption{A graph with three vertices.}
\label{fig:graph-three-vertices}
\end{figure}
Then the following hold:
\begin{enumerate}
\item
$\#_2(X,s)=\max\{ m+\ell,n \}$.
\item
$\chi^{\mathrm{qdl}}(X,s)=m+\ell$.
\end{enumerate}
Consequently, this quandle satisfies the Chen--Nagano-type inequality, 
but it does not satisfy the parity agreement in general.
\end{Ex}

\begin{proof}
The proof is similar to that of the preceding example.
We have
\[
X=\Z_m\sqcup\Z_n\sqcup\Z_\ell.
\]
For every $y\in\Z_n$, the point symmetry $s_y$ is the identity map on $X$.
On the other hand, for every $x\in\Z_m\sqcup\Z_\ell$, the point symmetry $s_x$ acts trivially on $\Z_m\sqcup\Z_\ell$ and as translation by $1$ on $\Z_n$.
Therefore, a maximal antipodal subset is either $\Z_m\sqcup\Z_\ell$ or $\Z_n$.
Hence
\[
\#_2(X,s)=\max\{m+\ell,n\}.
\]

Moreover, the displacement group acts trivially on $\Z_m\sqcup\Z_\ell$ and nontrivially only on $\Z_n$.
Thus, the fixed-point set of every nonidentity element of the displacement group is precisely
\[
\Z_m\sqcup\Z_\ell.
\]
It follows that
\[
\chi^{\mathrm{qdl}}(X,s)=m+\ell,
\]
which proves the second assertion.
\end{proof}

The preceding examples provide counterexamples to the parity agreement.
We next present counterexamples to the Chen--Nagano-type inequality using quandles constructed from 
more complicated graphs.

We first recall the notion of the line graph of a graph.
Let $G=(V,E)$ be an undirected simple graph.
The \textit{line graph} $L(G)$ is the graph whose vertex set is $E$, where two distinct vertices of $L(G)$ are adjacent if the corresponding edges of $G$ share a common endpoint.
We consider the case in which $G=K_5$, the complete graph on five vertices.
Its line graph is illustrated in Figure~\ref{fig:johnson}, 
which is also called the Johnson graph $J(5,2)$. 

\begin{figure}[htbp]
\centering
\includegraphics[width=0.4\textwidth]{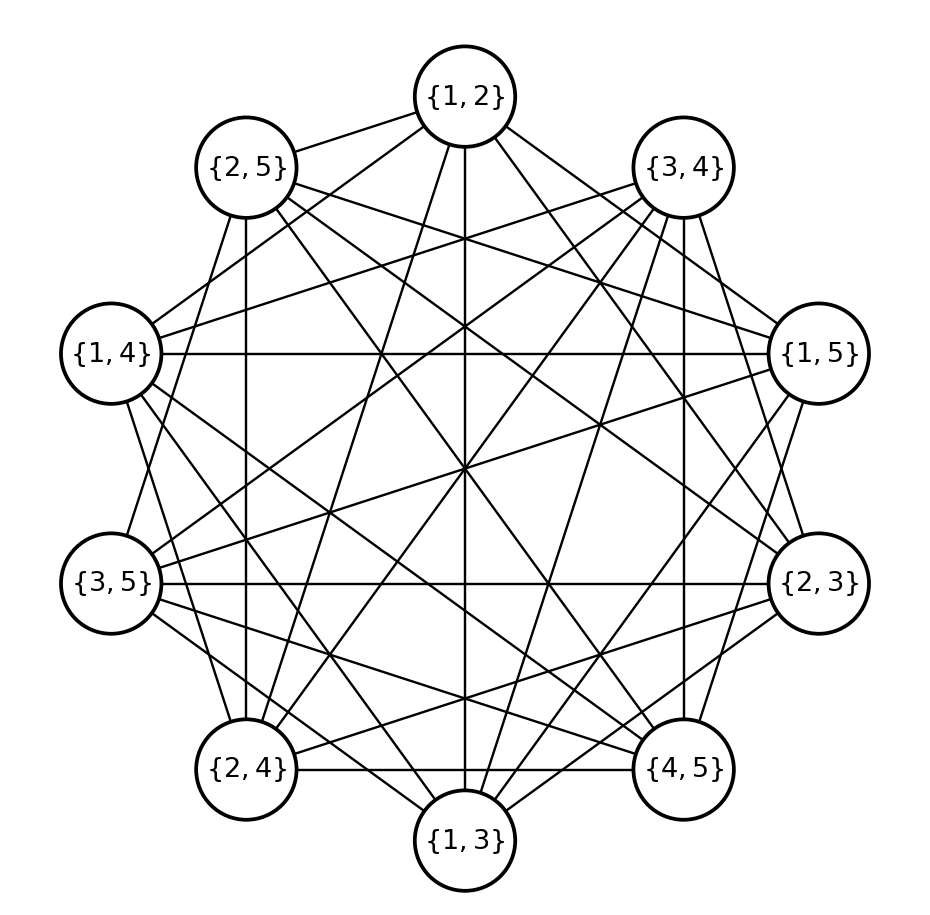}
\caption{The Johnson graph $J(5,2)$.}
\label{fig:johnson}
\end{figure}

\begin{Ex}
\label{ex:johnson} 
Let $K_5$ be the complete graph on five vertices, and let
\[
L(K_5)=(V,E)
\]
be its line graph.
We identify its vertex set with
\[
V=\{A_{i,j}\mid 1\leq i<j\leq 5\},
\]
where each $A_{i,j}$ is a copy of $\Z_2$.
We regard each edge of $L(K_5)$ as a pair of oppositely directed edges, each labeled by the nonzero element $1\in\Z_2$.
Then the quandle $(X,s)$ obtained from this labeled directed graph satisfies the following:
\begin{enumerate}
\item
$\#_2(X,s)=4$.
\item
$\chi^{\mathrm{qdl}}(X,s)=8$.
\end{enumerate}
Consequently, this quandle does not satisfy the Chen--Nagano-type inequality, but it does satisfy the parity agreement.
\end{Ex}

\begin{proof}
First, recall that the quandle $(X,s)$ is homogeneous; that is, its automorphism group acts transitively on $X$.
This follows from the vertex-transitivity of $L(K_5)$; see \cite{FT}.

We first prove~(1).
We claim that every maximal antipodal subset of $(X,s)$ is of the form
\[
A_{i,j}\sqcup A_{k,\ell},
\]
where $i,j,k,\ell$ are pairwise distinct.
Let $A$ be a maximal antipodal subset of $(X,s)$.
By homogeneity, we may assume without loss of generality that $A$ contains an element
\[
x\in A_{1,2}\cong\Z_2.
\]
Since $A$ is antipodal, every element of $A$ is fixed by $s_x$, and hence we have
\[
A\subset\Fix(s_x,X)
 =A_{1,2}\sqcup A_{3,4}\sqcup A_{3,5}\sqcup A_{4,5}.
\]
No two of the three components $A_{3,4}$, $A_{3,5}$, $A_{4,5}$ 
can be contained in the same antipodal subset, because the corresponding vertices of $L(K_5)$ are adjacent.
Consequently, $A$ contains at most one of these three components.
Since $A$ is maximal antipodal, it follows that
\[
A
 =A_{1,2}\sqcup A_{3,4},
 \qquad
A_{1,2}\sqcup A_{3,5},
 \qquad\text{or}\qquad
A_{1,2}\sqcup A_{4,5}.
\]
This proves the claim. 
In each case,
\[
\#A=\#\bigl(\Z_2\sqcup\Z_2\bigr)=4.
\]
Therefore,
\[
\#_2(X,s)=4.
\]

We next prove~(2).
For each $(a_1,\ldots,a_5) \in (\Z_2)^5$, define a map
\[
(a_1,\ldots,a_5) \colon X \longrightarrow X
\]
by
\[
(a_1,\ldots,a_5)(y) = y + a_i + a_j \in A_{i,j} 
\qquad 
(\mbox{if $y \in A_{i,j}$}) . 
\]
Since $A_{i,j}$ is isomorphic to $\Z_2$, the map
$(a_1,\ldots,a_5)$ fixes an element $y \in A_{i,j}$ if and only if
$a_i = a_j$.
This notation provides a convenient description of the point symmetries of $(X,s)$.
For example, if $x_{12} \in A_{1,2}$, then 
\[
s_{x_{12}} = (0,0,1,1,1),
\]
because both maps act trivially on
$A_{1,2} \sqcup A_{3,4} \sqcup A_{3,5} \sqcup A_{4,5}$ 
and act as translation by $1$ on each of the remaining components.
More generally, if $x_{ij} \in A_{i,j}$, then
\[
s_{x_{ij}} = (a_1,a_2,a_3,a_4,a_5),
\]
where $a_i = a_j = 0$ and $a_k = 1$ for every
$k \notin \{i,j\}$.
Indeed, $s_{x_{ij}}$ acts as translation by $1$ on each of the
components $A_{i,k}$ and $A_{j,k}$, where $k \notin \{i,j\}$,
and acts trivially on all other components. 
With this notation, the composition of two such maps corresponds to 
coordinatewise addition in $(\Z_2)^5$:
\[
(a_1,\ldots,a_5) \circ (b_1,\ldots,b_5)
=
(a_1+b_1,\ldots,a_5+b_5).
\]
In particular, each such map is an involution.

Using the above notions, we will determine the Euler characteristic. 
Let $x_{12}\in A_{1,2}$ and $x_{13}\in A_{1,3}$.
Since $s_{x_{13}}^{-1}=s_{x_{13}}$, the displacement group $\Dis(X,s)$ contains the element
\begin{align*}
s_{x_{12}}\circ s_{x_{13}}^{-1} =(0,0,1,1,1)\circ(0,1,0,1,1) =(0,1,1,0,0) . 
\end{align*}
This element acts trivially precisely on
\[
A_{1,4}\sqcup A_{1,5}\sqcup A_{2,3}\sqcup A_{4,5}.
\]
Therefore,
\begin{align*}
\chi^{\mathrm{qdl}}(X,s)
&\leq
\#\Fix\bigl(s_{x_{12}}\circ s_{x_{13}}^{-1},X\bigr)\\
&=
\#\bigl(
A_{1,4}\sqcup A_{1,5}\sqcup A_{2,3}\sqcup A_{4,5}
\bigr)\\
&=8.
\end{align*}

It remains to prove the reverse inequality.
Every element of $\Dis(X,s)$ can be represented by a vector
\[
(a_1,\ldots,a_5)\in(\Z_2)^5.
\]
Since the displacement group is generated by products of two point symmetries and every point symmetry is represented by a vector of odd weight, each element of $\Dis(X,s)$ has a representative of even weight.
Here, the weight of $(a_1,\ldots,a_5)$ is defined by
\[
\mathrm{wt}(a_1,\ldots,a_5)
:=\#\{i\mid a_i=1\}.
\]
It therefore suffices to consider representatives of weights $0$, $2$, and $4$.

If the weight is $0$, then the map is the identity, which may be excluded from consideration.

Suppose that the weight is $2$.
Then the situation is exactly similar to the case of $s_{x_{12}}\circ s_{x_{13}}^{-1}$. 
In this case the map acts as identity on $4$ components, and therefore fixes $8$ points. 

Suppose that the weight is $4$.
In this case, the map acts trivially on $6$ components and therefore has $12$ fixed points. 
For example, $(1,1,1,1,0)$ fixes pointwise the six components 
\[
A_{1,2},\quad
A_{1,3},\quad
A_{1,4},\quad
A_{2,3},\quad
A_{2,4},\quad
A_{3,4}.
\]

It follows that every nonidentity element of $\Dis(X,s)$ has at least eight fixed points.
Hence
\[
\chi^{\mathrm{qdl}}(X,s)\geq 8.
\]
Combining this with the opposite inequality obtained above, we complete the proof. 
\end{proof}

\begin{Rem}
The quandle in Example~\ref{ex:johnson} also appears in \cite{FT}.
Note that its underlying graph is bidirected and that all the associated
abelian groups are isomorphic to $\Z_2$.
As shown in \cite{FT}, this quandle can be realized as a subquandle of
the oriented real Grassmannian manifold $G_2(\R^5)^\sim$, 
consisting of oriented two-dimensional subspaces of $\R^5$. 
\end{Rem}

Finally, we conclude by posing several problems for future research related to the direction of this study.

\begin{Problem} \noindent 
\begin{enumerate} 
\item 
For each compact connected Riemannian symmetric space $M$, 
does there exists a finite subquandle $X$ of $M$ with the two-number and the same Euler characteristic? 
If exists, what properties such finite quandle satisfy? 
\item 
Find a condition for a finite quandle $(X,s)$ to satisfy the Chen-Nagano inequality 
$\#_2 (X,s) \geq \chi^{\mathrm{qdl}} (X,s)$. 
\item
Find a condition for a finite quandle $(X,s)$ to satisfy the parity agreement 
$\#_2 (X,s) \equiv \chi^{\mathrm{qdl}} (X,s)$ ($\mathrm{mod}\, 2$). 
\end{enumerate} 
\end{Problem}

It would be inetresting to develop a theory of finite quandles from its own. 
Moreover, it would be more challenging to connect the theories of symmetric spaces and quandles, 
for example, proving the properties or conjectures for symmetric spaces using theory of finite quandles. 

\section*{Acknowledgements}

The authors would like to express their sincere gratitude to Professors Ali Baklouti 
and Hideyuki Ishi, the organizers of the 8th Tunisian--Japanese Conference, 
for providing the opportunity for fruitful discussions. 
This paper grew out of discussions held during the conference. 

This work was partly supported by 
MEXT Promotion of Distinctive Joint Research Center Program JPMXP0723833165 
and Osaka Metropolitan University Strategic Research Promotion Project 
(Development of International Research Hubs). 
The second author was supported by JSPS KAKENHI Grant Number JP22K13919. 
The third author was supported by JSPS KAKENHI Grant Numbers JP23K22395 and JP24K21193.

\end{document}